\RequirePackage[english]{babel}
\documentclass{amsart}
\usepackage{amsthm,amsfonts,enumerate} 

\newcommand{\bP}{\mathbb{P}}

\newcommand{\bx}{\mathbf{x}}

\theoremstyle{plain}
\newtheorem{thm}{Theorem}[section]
\newtheorem*{thms}{Theorem}

\newtheorem{pr}[thm]{Proposition}
\newtheorem{co}[thm]{Corollary}

\newtheorem{lem}[thm]{Lemma}

\theoremstyle{remark}
\newtheorem{re}[thm]{Remark}
\newtheorem{ex}[thm]{Example}

\theoremstyle{definition}
\newtheorem{de}[thm]{Definition}

\numberwithin{equation}{section}

\DeclareMathOperator{\Tan}{Tan}
\DeclareMathOperator{\expdim}{expdim}
\DeclareMathOperator{\sym}{Sym}
\DeclareMathOperator{\Proj}{Proj}
\DeclareMathOperator{\im}{Im}
\DeclareMathOperator{\Car}{Car}

\DeclareMathOperator{\Hom}{Hom}

\begin{document}

\title{On varieties with higher osculating defect}
\author{Pietro De Poi}\email{pietro.depoi@uniud.it}
\address{Dipartimento di Matematica e Informatica, Universit\`a degli Studi di Udine, via delle Scienze, 206, 33100 Udine, Italy}
\author{Roberta Di Gennaro}\email{digennar@unina.it}
\author{Giovanna Ilardi}\email{giovanna.ilardi@unina.it}
\address{Dipartimento di Matematica e Applicazioni, Universit\`a degli Studi di Napoli ``Federico II'', 80126 Napoli, Italy}
\thanks{The first author was partially supported by  MiUR,
project ``Geometria delle variet\`a algebriche e dei loro spazi di moduli'' and by the University of Naples ``Federico II'', 
project ``F.A.R.O. 2010: Algebre di Hopf, differenziali e di vertice in geometria, 
topologia e teorie di campo classiche e quantistiche''.}


\begin{abstract}
In this paper, using the method of moving frames, we generalise some of Terracini's results on varieties with tangent defect.
In particular, we characterise varieties with higher order osculating defect in terms of Jacobians of higher fundamental forms and moreover
we characterise varieties with  ``small'' higher fundamental forms as contained in scrolls.
\end{abstract}

\maketitle

\section*{Introduction}

The starting point of this paper is given by the classical papers of Terracini \cite{T}, \cite{T1}, \cite{T2}, \cite{T3}
on the description of the $k$-dimensional varieties $V$  of $\bP^N(\mathbb C)$, $(N>2k)$, such that the embedded tangent variety $\Tan(V)$ is defective, i.e.
it has dimension less than $2k$ ($2k-\ell$ with $\ell>0$). In \cite{T1}, Terracini links this problem
to the determination of the linear systems of quadrics such that the  Jacobian matrix has rank $k-\ell$.
After Terracini, there have been further classical papers on this subject: here we cite only as an example 
\cite{De1},
\cite{M1}, \cite{M2}, \cite{M3}.


Terracini proved results bounding the tangent defect of $V$ and on
the structure of the varieties satisfying a certain number of Laplace equations. To state the results, we will say that
$V$ \emph{satisfies $\delta_s$ Laplace equations of order $s$} if---given a local parametrisation ${\mathbf{x}}(t_1,\dots,t_k)=(x_1(t_1,\dots,t_k),\dots,x_N(t_1,\dots,t_k))$ and denoted
 by ${\mathbf{x}}^I=\frac{\partial^{\vert I\vert}{\mathbf{x}}(t_1,\dots,t_k)}
{\partial t_1^{i_1}\dots\partial t_k^{i_k}}$ the partial derivatives of $\mathbf x$
---it
satisfies the following partial differential equations:
$$\sum_{0\leq\vert I\vert \leq s} E_I^{(h)}{\mathbf{x}}^{I} =0, \hspace{1cm} h=1,\dotsc, \delta_s$$ where
at least a $E_I^{(h)}\neq 0$ with $|I|=s$ and these equations are linearly independent.

In this paper, we apply the method of moving frames, developed by Darboux, Cartan  and others,
to understand the relationship between the algebraic geometry of subvarieties
of the projective space and their local projective differential geometry.
This was a project of the classical geometers,
revived by  Akivis and Goldberg (see \cite{AG} and references therein)  and Griffiths and Harris (see \cite{G-H})
and more recently by Landsberg (see for example \cite{L}, \cite{L1} and with other authors, see for example \cite{IL}, \cite{LR}).

We generalise Terracini's Theorem
to varieties with defect of higher order studying linear system of hypersurfaces (the fundamental forms)
instead of Laplace equations of every order satisfied by the variety. We prove the following.

\begin{thms}
Let $V \subseteq \bP^N$  be a  $k$-dimensional irreducible variety whose $t$-th fundamental form has dimension $k-\ell-1$, 
with $\ell >0$; then $V$ has $(t-1)$-osculating defect $\geq \ell$ and moreover it holds
\begin{enumerate}[1.]

 \item $V$ is contained in a $d$-dimensional scroll $S(\Sigma^h_r)$ in $\bP^r$, with $0 \leq h \leq k-\ell$ and $k-h\le r$.
  \item The tangent $\bP^{d}$'s  to $S(\Sigma^h_r)$ at the smooth points of a generic $\bP^r$ of $S(\Sigma^h_r)$ are contained in a linear space
of dimension $d_t-h=d_{t-1}+ k-\ell-h$---where $d_t$ is the dimension of the $t$-th osculating space to $V$ at its general point.
In particular, $r\le d\le d_{t-1}+ k-\ell-h$.
\end{enumerate}
\end{thms}
See Theorem \ref{mainthm}. 
\\ Moreover, we have obtained some classifications for the extremal cases of the preceding theorem; for example, we show
that, if $\ell=k-1$ and $t=2$, then $V$ is either a hypersurface or a developable $\bP^{k-1}$-bundle.

Successively, Terracini studied again varieties with tangent defect, but satisfying a number of Laplace equations less than  $\binom{k}{2} + l$
in \cite{T1}.

We generalise also this result as follows, in terms of fundamental forms:
\begin{thms}
Let $V \subseteq \bP^N$  be a  $k$-dimensional irreducible variety.
  $V$ has $t$-th osculating defect $o_t=\ell>0$ and the $(t+1)$-th fundamental form has dimension at least $k-\ell$ if and only if the
Jacobian matrix of the $(t+1)$-th fundamental form of $V$ has rank $k-\ell$.
\end{thms}
See Theorem \ref{thm terr gen}.

The article is structured as follows. In Section \ref{notation_sec} we give the basic notations and preliminaries, and we show some 
results that we need. Many of them either are natural generalisations of known results (mainly from \cite{G-H}) or are not very 
surprising; nevertheless, we think that including them can be useful by lack of adequate references.  
More precisely, after fixing some notations and recalling the basic definitions such as
Laplace equations, Darboux frames, the second fundamental form and apolarity, we prove the relation between the dimension of the second fundamental form
and the number of Laplace equations of order two for a $k$-dimensional
projective variety $V\subset\bP^N$: more precisely, if $V$ satisfies $\delta_2$ Laplace equations,
then the second fundamental form has dimension $\binom{k+1}{2}-1-\delta_2$.

Then, after recalling the definition of the osculating spaces
of every order, we link them to the higher fundamental forms proving in particular that the Jacobian system of the $t$-th fundamental form is contained in
the $(t-1)$-th fundamental form. We also prove the equivalence between the dimension of the $t$-th fundamental form
and the number of Laplace equations of order $t$, extending the above result for the second fundamental form.

We recall the definition of the $t$-th Gauss map and we show that its differential can be interpreted as the $t$-th fundamental form. Finally, we introduce
the definition of $t$-th dual variety of $V$ and we prove some lemmas about it.

In Section \ref{sec:terr} we state and prove the main theorems of the article, i. e. Theorems \ref{mainthm} and \ref{thm terr gen}. In order to do so, we also prove
a lemma on the tangent space of the higher osculating variety of $V$.

\section{Notation and preliminaries}\label{notation_sec}
We will follow notation as in \cite{G-H}, \cite{H}. Let $V \subset \bP^N$ be a projective variety of dimension $k$
over $\mathbb C$ that will be always irreducible. For any point $P\in V$ we use the following notation:
%
%
%
$\widetilde {T}_P(V)\subset \bP^N$
is the embedded tangent projective space to $V$ in $P$ and
$T_P(V)$ is the Zariski tangent space. 

With abuse of notation, we agree with \cite{G-H} by identifying the embedded tangent space in $\bP^N$ with its affine cone in $\mathbb C^{N+1}$ ,
in order to avoid any complicating of writing. With this interpretation $T_P(V)\cong \frac{\widetilde T_P(V)}{\mathbb C}$.
By $\mathbb G(N,t)$ we denote the Grassmannian of $t$-planes of $\bP^N$.
%
%
%
%
%
%
%

We define  $\Tan(V) :=\overline{\bigcup_{P \in V_0}\tilde {T}_P(V) }$ where $V_0\subset V$ is the smooth locus of $V$. $\Tan(V)$ has expected dimension $2k$,
and the case in which $\Tan(V)$ is less than expected was studied by many algebraic geometers:
classically Terracini (see \cite{T1}) linked the dimension $2k-\ell$ of $\Tan(V)$ with the number of Laplace equations that the variety $V$ satisfies,
and more recently Griffiths and Harris \cite{G-H} analysed the same dimension in terms of second fundamental form $II$.

Actually, for studying Laplace equations, it is usual to consider a parametric representation of $V$;
instead in \cite{G-H} and \cite{IL} the authors use the language of the Darboux frames. So, our first step is to understand in this language what means that
$V$ satisfies a Laplace equation.

We  begin to expose the definition of Laplace equations. Let  $V \subseteq \bP^N$ and let
$ {\mathbf x} = {\mathbf  x} (t_1 , \dotsc , t_k) = {\mathbf x} (\mathbf {t})$ be a local affine parametrisation of $V$ centred at the smooth point
$P=[p_0:p_1:\dotsb:p_N]$,
with---for example---$p_0\neq 0$ and ${\mathbf{x}}(\mathbf 0)=P$.
Let $I=(i_1,\dotsc,i_k)$ be a multiindex, that is a
$k$-tuple of non negative integers. We shall denote by $\vert I\vert$ the sum of the components of $I$, i.e.  $\vert I\vert=i_1+\dotsb+i_k$.
%
If ${\mathbf{x}}(t_1,\dots,t_k)=(x_1(t_1,\dots,t_k),\dots,x_N(t_1,\dots,t_k))$ is the above vector function,
we shall denote by ${\mathbf{x}}^I$ the partial derivatives of $\mathbf x$: 
$$
{\mathbf {x}}^I=\frac{\partial^{\vert I\vert}{\mathbf{x}}(t_1,\dots,t_k)}
{\partial t_1^{i_1}\dots\partial t_k^{i_k}}.
$$

\begin{de}\label{de:le}
By saying that $V$ \emph{satisfies $\delta_s$ Laplace equations of order $s$} we mean that---with the above local parametrisation $\mathbf{x}$
of $V$---$\mathbf{x}$ satisfies the following system of partial differential equations
\begin{align}\label{laplaceequations}
\sum_{0\leq\vert I\vert \leq s} E_I^{(h)}{\mathbf{x}}^{I} &=0, &E_I^{(h)}&\in\mathbb C,\  h =1,\dotsc, \delta_s
\end{align}
such that at least a $E_I^{(h)}\neq 0$ with $|I|=s$ and that these equations are linearly independent.
We say equivalently that $V$ represents the system of differential equations \eqref{laplaceequations} or  that $V$ is an integral variety for it.
\end{de}

It is not restrictive to suppose that $P=[1:0:\dotsb:0]$, ant therefore we have ${\mathbf{x}}(\mathbf 0)=P= \mathbf 0\in \mathbb A^N$, and that
\begin{align}
x_i(t_1,\dots,t_k)&=t_i & \forall i &\le k,
\end{align}
i.e. $x_{k+1}=\dotsb=x_N$ defines $\widetilde {T}_P(V)\subset \bP^N$. In these hypotheses, Equations \eqref{laplaceequations} become
\begin{align}\label{laplacecorto}
\sum_{2\leq\vert I\vert \leq s} E_I^{(h)}{\mathbf{x}}^{I} &=0 & h &=1,\dotsc, \delta_s.
\end{align}
In what follows, we will make this assumption.

At same time, to study the behaviour of $V$ in $P$, following \cite{G-H} (and references [2], [6], [7] and [10] therein) and \cite{L},
we consider the frame manifold on $V$, $\mathcal F (V)$; an element of it is a Darboux frame centred in $P$, that means an $(N+1)$-tuple
$$\left\{A_0;A_1,\ldots,A_k;\dotsc,A_N\right\}$$
which is a basis of $\mathbb C^{N+1}$ such that, if $\pi:\mathbb C^{N+1}\setminus\{0\}\rightarrow \bP^N$,
$$\pi(A_0)=P$$
and
\begin{equation*}
\pi(A_0), \pi(A_1),\ldots,\pi(A_k)  \textup{ span } \widetilde{T}_P(V).
\end{equation*}
Let us make this frame to move in  $\mathcal F (V)$;
then we can give structure equations (with Maurer-Cartan $1$-forms $\omega_i$, $\omega_{i,j}$ on $\mathcal F(\bP^N)$ restricted on $V$)
for the exterior derivatives of this moving frame

\begin{equation}\label{derivate_frame_eq}
\begin{cases}
 \omega_\mu=0 &\forall \mu>k\\
  d A_0=\sum_{i=0}^k \omega_i A_i \\
  d A_i=\sum_{j=0}^N \omega_{ij}A_j  &i=1,\dotsc,N\\
  d\omega_j= \sum_{h=0}^k \omega_h \wedge \omega_{h,j} & j=0,\dotsc,k\\
    d\omega_{ij}= \sum_{h=0}^N \omega_{i,h} \wedge \omega_{h,j}  &i=1,\dotsc,N,\ j=0,\dotsc,N. 
\end{cases}\end{equation}
\begin{re} Geometrically, the frame $\{A_i\}$ defines a coordinate simplex in $\bP^N$.
The $1$-forms $\omega_i, \omega_{ij}$ give the rotation matrix when the coordinate simplex is infinitesimally displaced; in particular, modulo $A_0$,
as $d A_0\in T^*_P(\bP^N)$ (the cotangent space), the $1$-forms $\omega_1,\dotsc,\omega_k$ give a basis for the cotangent space $T^*_P(V)$,
the corresponding $\pi(A_i)=v_i\in T_P(V)$ give a basis for $T_P(V)$ such that $v_i$ is tangent to the line $\overline{A_0A_i}$,
 and $\omega_{k+1} =\cdots =\omega_N=0$ on $T_P(V)$.
%
\end{re}
In such notation, we can define locally the second fundamental form:
\begin{de}\label{II}
The \emph{second fundamental form of $V$ in $P$} is the linear system $|II|$ in the projective space
$\bP(T_P(V))\cong \bP^{k-1}$ of the quadrics defined by the equations:
\begin{align*}
\sum_{i,j=1}^k q_{ij\mu} \omega_i\omega_j&=0, & \mu&=k+1,\ldots,N
\end{align*}
where the coefficients $q_{ij\mu}$ are defined by the relations
\begin{align}\label{eq:second}
\omega_{i\mu}&=\sum_{j=1}^k q_{ij\mu}\omega_j, & q_{ij\mu}&=q_{ji\mu}
\end{align}
obtained from $d\omega_\mu=0$, $\forall \mu>k$, via the Cartan lemma (see \cite[(1.17)]{G-H}).
\end{de}

We may symbolically write the second fundamental form as in \cite[(1.20)]{G-H} as
\begin{align}\label{eq_II}
d^2 A_0 \equiv \sum_{\substack{
0\le i,j\le k\\
k+1\le \mu \le N}
}q_{ij\mu}\omega_i \omega_j A_\mu& \mod \tilde T(V)
\end{align}
or, more intrinsically the following (global) map
\begin{align}\label{sym2}
  II\colon \sym^{(2)} T(V) &\rightarrow N(V)
\end{align}
where $N(V)$ is the normal bundle  ($N_P(V):=\dfrac{\mathbb C^{N+1}}{\tilde T_P(V)}$ as in \cite{G-H})
which in coordinates is
\begin{equation*}
   II(\sum_{i,j} a_{ij}v_i v_j)=\sum_{\substack{
0\le i,j\le k\\
k+1\le \mu \le N}
}
    q_{ij}a_{ij}A_\mu.
\end{equation*}

To relate the second fundamental form to the Laplace equations \eqref{laplaceequations}, for ease our exposition,
we consider the case $s=2$ . 
If there are $\delta_2$ independent relations of type:
\begin{align*}
\sum_{i, j=1}^k a_{i j}^{(\alpha)} {\mathbf x}^{(i j)} + \sum_{i=1}^k b_i ^{(\alpha)} {\mathbf x}^{(i)} + c^{(\alpha)}{\mathbf x} &= 0, & \alpha&=1,\dotsc,\delta_2,
\end{align*}
that, with our assumption on the coordinate become
\begin{align}\label{2nd}
\sum_{i, j=1}^k a_{i j}^{(\alpha)} {\mathbf x}^{(i j)}  &= 0, & \alpha&=1,\dotsc,\delta_2;
\end{align}
we can consider the linear system of quadrics of $\bP(T_P(V)^*)$ of dimension $\delta_2 -1$, generated by the quadrics of equations:
\begin{align}\label{eq_quadriche}
\sum_{i, j=1}^k a_{i j}^{(\alpha)} v_i  v_j &= 0 &\alpha &= 1, \dotsc,\delta_2;
\end{align}
it defines the linear system of quadrics \emph{associated} to the  system of Laplace equations.

We recall now some notions of apolarity. Since our definitions are base dependent, for ease our exposition,
we can say that two forms $f\in\mathbb C[x_0,\dotsc,x_N]$
and $g\in \mathbb C[y_0,\dotsc,y_N]= C[x_0,\dotsc,x_N]^*$ of the same degree $n$, are \emph{apolar} if
\begin{equation*}
\sum_{I=(i_0,\dotsc,i_N)} a_I b_I=0,
\end{equation*}
where $f=\sum_I a_I\mathbf x^I$ and $g=\sum_I b_I\mathbf y^I$.
Since $f$ and $g$ define hypersurfaces $F:=V(f)\subset \bP^N=\Proj(\mathbb C[x_0,\dotsc,x_N])$ and
$G:=V(g)\subset \bP^{N*}=\Proj(\mathbb C[y_0,\dotsc,y_N])$, we will say also that $F$ and $G$ are \emph{apolar} if $f$ and $g$ are apolar.

Given a system of hypersurface $H$ in $\bP^N$ we say that the linear system $K$ in $\bP^{N*}$
given by the hypersurfaces which are apolar to all the ones in $H$ is the \emph{apolar system} of $H$.

The following result is classical:

\begin{pr}\label{prop_IIapolar}
$\vert II\vert$ is the apolar system to  the system of quadrics \eqref{eq_quadriche}; so, if
$V$ satisfies $ \delta_2$ independent Laplace equations, then $ \dim |II| =\binom{k+1}{2}-1-\delta_2$.
\end{pr}
\begin{proof}
Since we can identify the parametrisation $\bx$ around $P$ with $\pi(A_0)$, then, by   \eqref{eq_II}
\begin{align*}
d^2 A_0 (\sum_{i, j=1}^k a_{i j}^{(\alpha)} v_i  v_j) &= \sum_{1\le i,j\le k}q_{ij\mu}  a_{i j}^{(\alpha)} & \alpha &= 1, \dotsc,\delta_2 & \mu &=k+1,\dotsc,N;
\end{align*}
for our choice of the coordinates. On the other hand,
\begin{align*}
d^2 A_0 (\sum_{i, j=1}^k a_{i j}^{(\alpha)} v_i  v_j)  &=\sum_{i, j=1}^k a_{i j}^{(\alpha)} \frac{d^2A_0}{d v_i  d v_j}  =\sum_{i, j=1}^k  a_{i j}^{(\alpha)} {\mathbf x}^{(i j)}
& \alpha &= 1, \dotsc,\delta_2.
\end{align*}


\end{proof}

The second fundamental form can be related also with the second osculating space that we define as follows

\begin{de}\label{T2}
Let $P\in V$, the \emph{second osculating space} to $V$ at $P$ is the subspace $\tilde T^{(2)}_P(V)\subset \bP^N$ spanned by $A_0$ and by all the
derivatives $\dfrac{d A_0}{d v_\alpha}=A_{\alpha}$ and $\dfrac{d A_\alpha}{d v_\beta}= \dfrac{d A_\beta}{d v_\alpha} $ for $1\leq\alpha,\beta\leq k$.
\end{de}
So from now on we can consider the Darboux frame
\begin{equation*}
\{A_0;A_1,\dotsc,A_k;A_{k+1},\dotsc,A_{k+r};A_{k+r+1},\dotsc,A_{N}\}
\end{equation*}
 so that $A_0;A_1,\dotsc,A_k;A_{k+1},\dotsc,A_{k+r}$ in $P$ span $\tilde T^{(2)}_P(V)$.
It is straightforward to see---for example from the proof of Proposition \ref{prop_IIapolar}---that
\begin{equation*}
\dim  |II|=r-1 \Longleftrightarrow \dim  \tilde T^{(2)}_P(V)=k+r.
\end{equation*}

Generalising Definition \ref{II}, we can define the $t$-th fundamental form and the
$t$-th osculating space at $P\in V$, for $t\geq 3$,  and relate them with \eqref{laplaceequations}.

\begin{de}\label{t^t def} Let $P\in V$, let $t\geq 3$ be an integer and $I=(i_1,\ldots,i_k)$ such that $|I|\leq t$.
The  \emph{$t-$th osculating space} to $V$ at $P$ is the subspace $\tilde T^{(t)}_P(V)\subset \bP^N$ spanned by $A_0$ and by all the
derivatives $\dfrac{d^{|I|}A_0}{d v_1^{i1}\dotsm d v_k^{i_k}}$, where $v_1,\dotsc, v_k$ span $T_P(V)$.

We will put
\begin{align*}
d_t&:= \dim(\tilde T^{(t)}_P(V)),\\
e_t&:= \expdim (\tilde T^{(t)}_P(V))=\min(N, d_{t-1}+\binom{k-1+t}{t}).
\end{align*}
\end{de}

\begin{re}
If  $V$ satisfies $\delta_t$ Laplace equations of order $t$, we have  $d_t = e_t - \delta_t$.
Moreover, since a Laplace equation of order $t$ contains at least one of the $\binom{k-1+t}{t}$ partial derivatives of order $t$, we have
$ \delta_t \le \binom{k-1+t}{t}$.

Put $k_t:=\binom{k+t}{t}-1$.  Obviously, $d_t\le \min(k_t,N)$.
If $N<k_t$, then $V\subseteq\bP^N$ represents at least $k_t-N$ Laplace
equations of order $t$. These Laplace equations are called \emph{trivial}.

\end{re}
\begin{de}
Let $t\geq 2$ and $V_0\subseteq V$ be the quasi projective variety of points where $\tilde T ^{(t)}_P (V)$ has maximal dimension.
The variety
\begin{equation*}
{\Tan}^t(V) := { \overline{\bigcup_{P \in V_0} \tilde T^{(t)}_P(V)}}
\end{equation*}
  is called  the \emph{variety of  osculating $t$-spaces to $V$}.
Its expected dimension is
\begin{equation*}
\expdim {\Tan}^t(V):=\min(k+ d_t,N)
\end{equation*}
The \emph{$t$-th osculating defect of $V$} is the integer
\begin{equation*}
o _t:=\expdim{\Tan}^t(V) -\dim {\Tan}^t(V).
\end{equation*}
 If $t=1$, we call $o_1$ the \emph{tangent defect}.
\end{de}

\begin{re}
Obviously we have
\begin{equation*}
d_t\le d_{t-1}+ \binom{k-1+t}{t}\le \dotsm \le \sum_{i=1}^t\binom{k+i-1}{i}= k_t.
\end{equation*}
\end{re}
We will study the osculating defects related to the fundamental forms. Following \cite{G-H} and recalling \eqref{eq_II} we give
\begin{de}\label{ff_def}
The \emph{$t-$th fundamental form} of $ V$ in $P$
 is the linear system $|I^t|$ in the projective space $\bP (T_P(V))\cong \bP^{k-1}$ of hypersurfaces of degree $t$ defined symbolically by the equations:
\begin{equation*}
d^t A_0=0;
\end{equation*}
more intrinsically, we write $I^t$ as the map
\begin{equation*}
   I^t\colon  \sym^{(t)} T(V) \rightarrow N^t(V)
\end{equation*}
where $N^t(V)$ is the bundle defined locally as $N^t_P(V):=\displaystyle{\frac{\mathbb C^{N+1}}{{\tilde T}^{(t-1)}_P(V)} }$
and the map $I^t$ is defined locally on each $v\in T_P(V)$ as
\begin{equation*}
     v^t \mapsto \displaystyle{\frac{d^t A_0}{d v^t} }\mod {\tilde T}^{(t-1)}_P(V).
\end{equation*}
\end{de}
Choose a Darboux frame
\begin{equation}\label{eq:darbu}
\{A_0;A_1,\dotsc,A_k;A_{k+1},\dotsc,A_{d_2};A_{d_2+1},\dotsc,A_{d_s};\dotsc,A_{d_t};\dotsc, A_N\}
\end{equation}
such that $A_0,A_1,\dotsc,A_{d_s}$ span ${\tilde T}^{(s)}_P(V)$ $\forall s=1,\dotsc, t$, with $d_1:=k$.
By the definition of  ${\tilde T}^{(s)}_P(V)$, we have that
\begin{align}\label{eq:dsup}
d A_{\alpha_{s-1}}& \equiv 0 \mod  \tilde T^{(s)}_P(V) & \alpha_{s-1}&=d_{s-2}+1,\dotsc, d_{s-1}, &  s&=2,\dotsc, t-1,
\end{align}
where we put $d_0=0$, from  \eqref{derivate_frame_eq} we have
\begin{align}\label{eq:omega}
\omega_{\alpha_{s-1},\mu_s}&=0 &  \alpha_{s-1}&=d_{s-2}+1,\dotsc, d_{s-1},\ \mu_s>d_s & s&=2,\dotsc t-1,
\end{align}
from which we infer, after some computations
\begin{equation}\label{eq:tfond}
d^ t A_0 \equiv
\sum_{\substack{d_{s-1}+1\le \alpha_s \le d_{s}\\
s=1,\dotsc, t-1\\
d_{t-1}+1\le \alpha_{t}\le N}}
\omega_{\alpha_1}\omega_{\alpha_1,\alpha_2}\dotsm \omega_{\alpha_s,\alpha_{s+1}}\dotsm \omega_{\alpha_{t-1},\alpha_{t}}A_{\alpha_{t}} \mod {\tilde T}^{(t-1)}_P(V)
\end{equation}
or, using Cartan lemma,
\begin{equation}\label{eq:tfond1}
\begin{array}{ll}d^ t A_0 & \equiv 
\displaystyle{\sum_{\substack{
1\le i_1,\dotsc,i_t\le k\\
d_{t-1}+1\le \alpha_t\le N
}} q_{i_1,\dotsc,i_t,\alpha_t}\omega_{i_1}\dotsm\omega_{i_t} A_{\alpha_t}=}\\ &
=\displaystyle{\sum_{\substack{
|I|=t\\
d_{t-1}+1\le \alpha_t\le N
}} q_{I,\alpha_t}\omega_I A_{\alpha_t}
\ \mod {\tilde T}^{(t-1)}_P(V)},
\end{array}
\end{equation}
with the natural symmetries for the indices $i_1,\dotsc,i_t$ of  $q_{i_1,\dotsc,i_t,\alpha_t}$ which can be expressed as
\begin{align}\label{eq:eulero}
\frac{d^{t-1}A_i}{d v_j}&\equiv \frac{d^{t-1}A_j}{d v_i} &\mod &{\tilde T}^{(t-1)}_P(V), & i,j=1,\dotsc,k.
\end{align}

From \eqref{eq:omega}
\begin{align*}
0&=d\omega_{\alpha_{s-1},\mu_s} =\sum_{h_s=d_{s-1}+1}^{d_s}\omega_{\alpha_{s-1},h_s}\wedge\omega_{h_s,\mu_s}    \\  \alpha_{s-1}&=d_{s-2}+1,\dotsc, d_{s-1},\ \ \ \ \mu_s>d_s , \ \ \ \  s=2,\dotsc t-1.
\end{align*}
Now, $\omega_{\alpha_{s-1},h_s}$ and $\omega_{h_s,\mu_s}$ are horizontal for the fibration $\tilde T^{(t-1)}_P(V)\to V$, and therefore---by induction on $s$, since
the case $s=2$ is in \cite[page 374]{G-H}---they are a linear combination of $\omega_1,\dotsc,\omega_k$; then, we have
\begin{multline*}
0=d\omega_{\alpha_{s-1},\mu_s}(\frac{\partial}{\partial \omega_\gamma})=\sum_{h_s=d_{s-1}+1}^{d_s}(\frac{\partial \omega_{\alpha_{s-1},h_s}}{\partial \omega_\gamma}
\omega_{h_s,\mu_s}-
\omega_{\alpha_{s-1},h_s}\frac{\partial \omega_{h_s,\mu_s}}{\partial \omega_\gamma})   \\
\alpha_{s-1}=d_{s-2}+1,\dotsc, d_{s-1},\ \mu_s>d_s\\
\gamma=1,\dotsc,k,\ s=2,\dotsc t-1.
\end{multline*}
which means
\begin{multline}\label{eq:tanti}
\sum_{h_s=d_{s-1}+1}^{d_s}(\frac{\partial \omega_{\alpha_{s-1},h_s}}{\partial \omega_\gamma} \omega_{h_s,\mu_s})
=\sum_{h_s=d_{s-1}+1}^{d_s}(\frac{\partial \omega_{h_s,\mu_s}}{\partial \omega_\gamma}\omega_{\alpha_{s-1},h_s})   \\
\alpha_{s-1}=d_{s-2}+1,\dotsc, d_{s-1},\ \mu_s>d_s\\
\gamma=1,\dotsc,k,\ s=2,\dotsc t-1.
\end{multline}
Since the linear system $|I^t|$ is generated, from Relation \eqref{eq:tfond}, by the following polynomials of degree $t$
\begin{align*}
V_{\alpha_t}&:=\sum_{\substack{d_{s-1}+1\le \alpha_s \le d_{s}\\
s=1,\dotsc, t-1}}
\omega_{\alpha_1}\omega_{\alpha_1,\alpha_2}\dotsm \omega_{\alpha_s,\alpha_{s+1}}\dotsm \omega_{\alpha_{t-1},\alpha_t}, & d_{t-1}+1&\le \alpha_{t}\le N
\end{align*}
then, we can prove Theorem \ref{thm:I}; in order to do so, we recall that

\begin{de}
Let $\Sigma$  be the linear system of dimension $d$ of hypersurfaces of degree $n$ ($n>1$) in  $\bP^N$ ($N>1$), generated
by the $d+1$ hypersurfaces $f_0=0, \dotsc, f_d=0$.

The Jacobian matrix of the forms $f_0, \dotsc, f_d$,
\begin{equation*}
J(\Sigma):=(\partial f_i/\partial x_j)_{i=0, \dotsc,d; j=0,\dotsc,r}
\end{equation*}
is said the \emph{Jacobian matrix} of the system $\Sigma$.

The \emph{Jacobian system} of $\Sigma$  is the linear system of the minors of maximum order of $J(\Sigma)$. Obviously, the Jacobian system does not
depend on the choice of $f_0, \dotsc, f_d$, but only on $\Sigma$.
\end{de}

\begin{thm}\label{thm:I}
Given a $k$-dimensional projective variety $V\subset\bP^N$, its $t$-th fundamental  form  $|I^t|$ is a linear system of polynomials of degree $t$
whose Jacobian system is contained in the  $(t-1)$-th fundamental  form  $|I^{t-1}|$.
\end{thm}
\begin{proof}
With notation as above, we start considering, with $d_{t-1}+1\le \alpha_{t}\le N$,
\begin{multline*}
\frac{\partial V_{\alpha_t}}{\partial \omega_\gamma}=\sum_{\substack{d_{s-1}+1\le \alpha_s \le d_{s}\\
s=2,\dotsc, t-1}}
\omega_{\gamma,\alpha_2}\dotsm \omega_{\alpha_s,\alpha_{s+1}}\dotsm
\omega_{\alpha_{t-1},\alpha_t}+\dotsb\\
\dotsb +\sum_{\substack{d_{s-1}+1\le \alpha_s \le d_{s}\\
s=1,\dotsc, t-1}}
(\omega_{\alpha_1}\dotsm \frac{\partial \omega_{\alpha_s,\alpha_{s+1}}}{\partial \omega_\gamma}\dotsm
\omega_{\alpha_{t-1},\alpha_t}+\dotsb\\
\dotsb +\omega_{\alpha_1}\dotsm \omega_{\alpha_s,\alpha_{s+1}}\dotsm
\frac{\partial\omega_{\alpha_{t-1},\alpha_t}}{\partial \omega_\gamma});
\end{multline*}
then from \eqref{eq:tanti}, we deduce
\begin{multline*}
\frac{\partial V_{\alpha_t}}{\partial \omega_\gamma}=\sum_{\substack{d_{s-1}+1\le \alpha_s \le d_{s}\\
s=2,\dotsc, t-1}}
\omega_{\gamma,\alpha_2}\dotsm \omega_{\alpha_s,\alpha_{s+1}}\dotsm
\omega_{\alpha_{t-1},\alpha_t}+\\
+(t-1)\sum_{\substack{d_{s-1}+1\le \alpha_s \le d_{s}\\
s=1,\dotsc, t-1}}
\omega_{\alpha_1}\dotsm \omega_{\alpha_s,\alpha_{s+1}}\dotsm
\frac{\partial\omega_{\alpha_{t-1},\alpha_t}}{\partial \omega_\gamma};
\end{multline*}
then, for example from \eqref{eq:second}
\begin{equation*}
\omega_{\gamma,\alpha_2}=\sum_{\alpha_1=1}^k q_{\gamma,\alpha_1,\alpha_2}\omega_{\alpha_1}= \sum_{\alpha_1=1}^k q_{\alpha_1,\gamma,\alpha_2}\omega_{\alpha_1}=
\sum_{\alpha_1=1}^k\frac{\partial \omega_{\alpha_1,\alpha_2}}{\partial \omega_\gamma}\omega_{\alpha_1}
\end{equation*}
and again from \eqref{eq:tanti},
\begin{equation*}
\frac{\partial V_{\alpha_t}}{\partial \omega_\gamma}=t\sum_{\substack{d_{s-1}+1\le \alpha_s \le d_{s}\\
s=1,\dotsc, t-1}}
\omega_{\alpha_1}\dotsm \omega_{\alpha_s,\alpha_{s+1}}\dotsm
\frac{\partial\omega_{\alpha_{t-1},\alpha_t}}{\partial \omega_\gamma}.
\end{equation*}
\end{proof}

Actually, as for the second fundamental form, Proposition \ref{prop_apolarity} holds, with an adapted proof from the one as in Proposition \ref{prop_IIapolar}.
In order to do so, we suppose to fix a Darboux frame as \eqref{eq:darbu}; then, with this choice, if we have a system of $\delta_t$ Laplace
equations of order $t$ as in \eqref{laplaceequations}, they can be expressed as
\begin{align}\label{laplaceequationsgen}
\sum_{|I|= t} E_I^{(h)}{\mathbf{x}}^{I} &=0 & h &=1,\dotsc, \delta_t;
\end{align}
then, as in \eqref{eq_quadriche} we can define the linear systems of homogeneous polynomials of degree $t$ \emph{associated} to
\eqref{laplaceequationsgen}:
\begin{align}\label{asso}
\sum_{|I|= t} E_I^{(h)}{\mathbf{v}}_{I} &=0 & h &=1,\dotsc, \delta_t,
\end{align}
where $\mathbf v_I=\prod_{\substack{i=1,\dotsc,k\\ i_1+\dotsm+ i_k=t}}
v_i^{i_j}.$

\begin{pr}\label{prop_apolarity}
If $V$ satisfies $\delta_t$ Laplace equations of order $t$ like in \eqref{laplaceequations},
the $t$-th fundamental form is the apolar system to  the system of the hypersurfaces of degree $t$ associated to the system of Laplace equations (i. e.
the hypersurfaces in \eqref{asso}), and vice versa.
\end{pr}
\begin{proof}
It is enough to repeat the proof of Proposition \ref{prop_IIapolar} with an adapted local coordinate system.
More precisely, we can choose a Darboux frame as in \eqref{eq:darbu}.
Since we can identify the parametrisation $\bx$ around $P$ with $\pi(A_0)$, then, by our hypothesis, Laplace equations of order $t$ become
\begin{align}
\sum_{|I| =t} E_I^{(h)}{\mathbf{x}}^{I} &=0 & h &=1,\dotsc, \delta_t.
\end{align}
By  \eqref{eq:tfond1}, we have
\begin{align*}
d^ t A_0 (\sum_{|I| =t} E_I^{(h)}{\mathbf{v}}_{I})&= \sum_{|I| =t}  q_{I,\beta} E_I^{(h)} &  h &=1,\dotsc, \delta_t & \beta&=d_{t-1}+1,\dotsc, N,
\end{align*}
and, on the other hand,
\begin{align*}
d^ t A_0 (\sum_{|I| =t} E_I^{(h)}{\mathbf{v}}_{I})&= \sum_{|I| =t}  E_I^{(h)}\frac{d^t A_0}{(d\mathbf{v})^I}=\sum_{|I| =t} E_I^{(h)}{\mathbf{x}}^{I} & h &=1,\dotsc, \delta_t.
\end{align*}
\end{proof}
From Proposition \ref{prop_apolarity} we recover immediately the following
\begin{co}\label{cor_dimensione}
If $V$ satisfies $\delta_t$ Laplace equations of order $t$ like in \eqref{laplaceequations}, 
the $t$-th fundamental form has dimension $\binom{k-1+t}{t}-1-\delta_t$,
and vice versa.
\end{co}
We will denote the dimension of the $t$-th fundamental form as $\Delta_t$:
\begin{equation*}
\Delta_t:=\dim(|I^t|).
\end{equation*}

\begin{co}\label{cor_dimensione ff}
If $N\ge k_t$, we have that
\begin{equation*}
d_t=  d_{t-1}+\Delta_t+1,
\end{equation*}
and vice versa, if $d_t=  d_{t-1}+\Delta+1$, then the $t$-th fundamental form has dimension $\Delta$.
\end{co}


From now on, we will suppose that our Darboux frame is as in \eqref{eq:darbu}.

In order to prove the results of the following section, we recall also the following notation and definitions.

Let $\Sigma^h_t \subset \mathbb{G}(N,t)$ be a subvariety of pure dimension $h$. Let
 $I_{\Sigma^h_t}\subset \Sigma^h_t \times \bP^N$ be the incidence variety of the pairs $(\sigma, q)$ such that
$q\in \sigma$ and let $p_1\colon I_{\Sigma^h_t} \to \Sigma^h_t$ and $p_2\colon I_{\Sigma^h_t} \to \bP^N$ be the maps induced by restricting
to $I_{\Sigma^h_t}$ the canonical projections of $\Sigma^h_t \times \bP^N$ to its factors.

The morphism  $p_1\colon I_{\Sigma^h_t} \to \Sigma^h_t$ is said to be a \emph{family of $t$-dimensional linear subvarieties of $\bP^N$}.
$\Sigma^h_t$ is the parameter space of the family, but for brevity we will often refer to it as to the family itself.
Obviously,
\begin{equation*}
\dim(I_{\Sigma^h_t})= t + \dim(\Sigma^h_t).
\end{equation*}
Let us suppose that $\Sigma^h_t$ is irreducible. We will denote by $S(\Sigma^h_t)$ the image of $I_{\Sigma^h_t}$ under $p_2$.
$S(\Sigma^h_t)$ is---by definition---a \emph{scroll in $\bP^r$} of $\bP^N$.
The previous notation will be useful to study osculating variety.
\begin{de}
Let $t\geq 1$, the \emph{$t$-th projective Gauss map} is the rational map
\begin{align*}
    \gamma^t\colon &V \dashrightarrow \mathbb G(\bP^N,d_t)\\
    & P \mapsto {\tilde T}^{(t)}_P(V).
\end{align*}
\end{de}
\begin{re}\label{rem_dimension}
The \emph{$t$-th osculating variety} is $\widetilde{{\Tan}^t}(V) = {\overline{\bigcup_{P \in V_0}\gamma^t(P)}}\subset \mathbb G(k_t,\bP^N)$
where, as before, $\ V_0$ denotes the open subset of $\ V$ of the points for which $\ \dim {\tilde T}^{(t)}_P(V)=d_t$
and then ${\Tan}^t(V)$ is the scroll $S(\widetilde{{\Tan}^t}(V))$ of dimension
\begin{equation*}
\dim{\Tan}^t(V)\leq\dim \im \gamma^t+d_t =k+d_t-\dim((\gamma^t)^{-1}(\Pi)),
\end{equation*}
where $\Pi$ is a general element of  $\widetilde{{\Tan}^t}(V)$.
\end{re}

We prove now
\begin{thm}\label{13}
The first differential of $\gamma^t$ at $P$ is the $(t+1)$-th fundamental form at $P$.
\end{thm}

\begin{proof}
We have, by the definition of $\gamma^t$, that
\begin{equation*}
    d\gamma^t_P\colon T_P V \dashrightarrow T_{\tilde T^{(t)}_P V} \mathbb G(\bP^N,d_t),
\end{equation*}
and we recall that  $T_{\tilde T^{(t)}_P V} \mathbb G(\bP^N,d_t)\cong \Hom(\tilde T^{(t)}_P V, N^{t+1}_P(V))$; moreover
if we choose a Darboux frame as in \eqref{eq:darbu}, we have that $d A_0\in \tilde T_P V \subset \tilde T^{(t)}_P V$ and
\begin{equation*}
\frac{\tilde T^{(t)}_P V}{\mathbb C A_0}=T^{(t)}_P V
\end{equation*}
and therefore  $d\gamma^t_P\in \Hom (T_P V\otimes T^{(t)}_P V,  N^{t+1}_P(V)) $.

Now, we remark that, in our Darboux frame, we can interpret $\gamma^t$ as
\begin{equation*}
 \gamma^t(P)=A_0\wedge\dotsb \wedge A_{d_t},
\end{equation*}
and therefore by \eqref{derivate_frame_eq},
\begin{align*}
 d\gamma^t_P&\equiv\sum_{\substack{
1\le i \le d_t\\
d_t+1\le j \le N
}}(-1)^{d_t-i+1} \omega _{i,j} A_0\wedge\dotsb\wedge \hat {A_i}\wedge\dotsb \wedge A_{d_t}\wedge A_{j}, & \mod \tilde T^{(t)}_P V;
\end{align*}
now, a basis for $T_P V\otimes T^{(t)}_P V$ can be expressed by $(A_\alpha\otimes A_\mu)_{\substack{\alpha=1,\dotsc,k\\ \mu=1,\dotsc, d_t}}$,
and
\begin{align*}
 d\gamma^t_P(A_\alpha\otimes A_\mu) &=\sum_{
d_t+1\le j \le N} \omega _{\mu,j}(A_\alpha) A_{j}\in  N^{t+1}_P(V)
\end{align*}
on the other hand, for the $(t+1)$-th fundamental form we have
\begin{align*}
\frac{d A_\mu}{d v_\alpha} &\equiv \sum_{\substack{
d_t+1\le j \le N
}} \omega _{\mu,j}(A_\alpha) A_{j} &  \mod \tilde T^{(t)}_P(V).
\end{align*}

\end{proof}





We recall now the definition of higher order dual varieties (see \cite{pi}), which is the natural extension of the definition of the dual variety:
\begin{de}
Let $V\subset\bP^N$ be a projective variety; for \emph{$t$-th dual variety} of $V$, $\check V^{(t)}$,  we mean
\begin{equation}\label{tan}
\check V^{(t)}=\overline{\bigcup_{P\in V_0} C_P^{(t)}(V)}
\end{equation}
where---as before---$V_0\subset V$ is the set of the points for which $\dim T^{(t)}_P(V)=d_t$ and $C_p^{(t)}(V)$ is
\begin{equation*}
C_P^{(t)}(V):=\bigcap_{K\in T_P^{(t)}(V)} K= \{H\in{\bP^N}^* \mid H\supset \tilde T^{(t)}_P(V)\} \subset {\bP^N}^*.
\end{equation*}
and it is classically called the \emph{$t$-th characteristic space} of $V$ in $P$.
\end{de}
We can now make observation similar to the ones in \cite[\S 3(a)]{G-H}: the elements of $C_P^{(t)}(V)$ can be naturally identified with
the hyperplanes in $\bP(N_P^{t+1}(V))$ and therefore $\check V^{(t)}$ is just the image of the map
\begin{equation*}
\delta^{t}\colon \bP(N^{t+1}(V)^*)\to  {\bP^N}^*,
\end{equation*}
analogous to the one of \cite[(3.1)]{G-H}.
In term of the frames, a hyperplane $\xi$ of $\bP(N^{t+1}_P(V))$ can be given by choosing $A_{d_t+1}, \dotsc, A_{N-1}$ such that
their projection in  $N_P^{t+1}(V)=\mathbb C^{N+1}/\tilde T^{(t)}_P(V)$ spans $\xi$. Therefore, in term of coordinates, $\delta^t$ can be expressed as
\begin{equation*}
\delta^t(P,\xi)=A_0\wedge A_1\wedge \dotsb \wedge A_{N-1},
\end{equation*}
(see  \cite[(3.2)]{G-H}) or, if we choose dual coordinates
\begin{equation*}
A_i^*:=(-1)^{N-i}A_0\wedge \dotsb \wedge A_{i-1}\wedge A_{i+1}\wedge \dotsb \wedge A_N,
\end{equation*}
$\delta^t(P,\xi)=A_N^*$.
From relations \eqref{derivate_frame_eq} we deduce
\begin{equation*}
d A_j^*= \sum_{i\neq j} (-\omega_{i,j}A_i^*+ \omega_{i,i}A_j^*),
\end{equation*}
and in particular
\begin{equation*}
d A_N^*= \sum_{i=0}^{N-1} (-\omega_{i,N}A_i^*+ \omega_{i,i}A_N^*)=\sum_{i=1}^{N-1} (-\omega_{i,N}A_i^*)+(\omega_0 +\sum_{i=1}^{N-1} \omega_{i,i})A_N^*.
\end{equation*}
By the definition of  $\check V^{(t)}$, for its dimension we have
\begin{equation*}
N-d_t-1\le \dim\check V^{(t)}=:d_{t,1}\le N-d_t-1+k.
\end{equation*}

If we choose a Darboux frame as in \eqref{eq:darbu}, these formulas become, thank 
to \eqref{eq:omega}
\begin{align*}
d A_j^* &= \sum_{\substack{i\neq j\\ i>u-2 }}(-\omega_{i,j}A_i^*+ \omega_{i,i}A_j^*),
&\begin{cases}
j=d_{u-1}+1,\dotsc, d_{u}&\textup{if  $u=0,\dotsc, t-1$,}\\
j=d_{t-1}+1,\dotsc,N& \textup{if $u=t$},
\end{cases}
\end{align*}
where we put $d_{-1}:=-1$ when we vary $j$; in particular,
\begin{equation*}
d A_N^* = \sum_{i= t-1}^{N-1}(-\omega_{i,N}A_i^*+ \omega_{i,i}A_N^*))=\sum_{i=t-1}^{N-1} (-\omega_{i,N}A_i^*)+(\omega_0 +\sum_{i=t-1}^{N-1} \omega_{i,i})A_N^*,
\end{equation*}
and therefore
\begin{align}\label{eq:deg}
d A_N^* &\equiv \sum_{i=t-1}^{N-1} (-\omega_{i,N}A_i^*)&\mod  A_N^*.
\end{align}
\begin{de}
We say that $\check V^{(t)}$ is \emph{degenerate}, if it has dimension less than expected: $d_{t,1}< N-1-d_{t}+k$.
\end{de}

In Relation \eqref{eq:deg} the last $N-d_t-1$ forms $\omega_{i,N}$, $i=d_t+1,\dotsc,N-1$, restrict to a basis for the forms of the fibres
$\bP(N^{t+1}_P)^*=\bP^{N-1-d_t}$; in fact, they describe the variation of $\xi$ when $P$ is held fixed.
The first $\omega_{i,N}$, with $i\le d_t$ are horizontal for the fibreing  $\bP(N^{t+1}(V)^*)\to V$, and therefore
 $\check V^{(t)}$ is degenerate if and only if
\begin{align*}
\omega_{i_1,N}\wedge\dotsb \wedge \omega_{i_k,N} &=0&   \forall i_1, \dotsc i_k\ \textup{with}\  t-1\le i_1 < \dotsb < i_k \le d_t.
\end{align*}

Now, if we put $d_{t,s}:=\dim \tilde T^{(s)}_\xi(\check V^{(t)})$, if $N-d_{t,s}\ge d_{t-1}+1$, i.e. $d_{t,s}\le N-d_{t-1}-1$
(otherwise  $\tilde T^{(s)}_\xi(\check V^{(t)})=\bP^{N*}$),
we can choose a Darboux frame such that  $T^{(s)}_\xi(\check V^{(t)})$ is generated by $A_N^*$ and $A_{N-1}^*,\dotsc,A_{N-d_{t,s}}^*$.


Let us now define the characteristic varieties of a projective variety $V\subset\bP^N$.


\begin{de}
The \emph{variety of the $s$-th characteristic spaces of the $t$-th osculating spaces} of $V$ is the $s$-th dual of the $t$-th dual variety of $V$, that is
\begin{equation*}
\Car^s_t(V):=\overline{\bigcup_{\xi \in \check V^{(t)}_0} C_\xi^{(s)}(\check V^{(t)})}\subset\bP^N
\end{equation*}
where $\check V^{(t)}_0$ is the open subset of  $\check V^{(t)}$ of the points $\xi$ are such that $\xi\supset {\tilde T}^{(t)}_P(V)$ and
$\dim  {\tilde T}^{(t)}_P(V)=d_t$;
we will denote in the following this  $s$-th characteristic space of $\check V^{(t)}$ in a general $\xi\supset \tilde T_P^{(t)}(V)$ as
$C^{(s)}_{t,P}(V):= C_\xi^{(s)}(\check V^{(t)})$; then, using the above notation, $\dim (C^{(s)}_{t,P}(V))=N-1-d_{t,s}$.
\end{de}
\begin{lem}\label{lemma}
With notations as above, if $P\in V$, we have
\begin{enumerate}[1.]
\item\label{a} $\tilde T^{(t-1)}_P(V) \subset C^{(1)}_{t,P}(V)$;
\item\label{b} $P\in  C^{(s)}_{t,P}(V)$;
\item\label{c} if $\xi\in C^{(t)}_P(V)$, $\tilde T_\xi(\check V^{(t)}) \subset C^{(t-1)}_P(V)$.
\end{enumerate}
\end{lem}

\begin{proof}
\eqref{a}: using the above notations, we have that, since  $\tilde T_\xi(\check V^{(t)})$ is generated by $A_{N}^*$ and $A_{N-1}^*,\dotsc,A_{t-1}^*$,
and we can choose a frame such that we have that the first $d_{t,1}$, $A_{N-1}^*,\dotsc,A_{N-d_{t,1}}^*$ are a base of  $\tilde T_\xi(\check V^{(t)})$.
Then, $C_\xi^{(1)}(\check V^{(t)})$ contains $A_0$ and $A_1,\dotsc,A_{N-d_{t,1}-1}$, and since $d_{t-1}\le d_t-k\le  N-d_{t,1}-1$, we have the assertion.

\eqref{b}: since $\ \tilde T^{(s)}_\xi(\check V^{(t)})\ $ is generated, in an appropriate frame, by $\ A_N^*\ $ and $\ A_{N-1}^*,\dotsc,A_{{N-d_{t,s}}}^*$, we have that
$C_\xi^{(s)}(\check V^{(t)})$ contains $A_0$.

\eqref{c}: it is just \eqref{a} in the dual space.
\end{proof}

\begin{co}\label{cor:lem}
With notations as above, if $P\in V$, $\xi \in \check V^{(t)}$ and $Q\in C^{(s)}_{t,P}(V)$ are general points, then $\tilde T_\xi (\Car^s_t(V))\subset C^{(s-1)}_{t,P}(V)$.
\end{co}

\begin{proof}
It is simply the dual of Lemma \ref{lemma}, \eqref{c}.
\end{proof}
\section{Terracini's theorems and generalisations}\label{sec:terr}
%
%
%
%
%
%
In this section we generalise classical Terracini's results in terms of osculating defect and higher fundamental forms instead of Laplace equations,
so that we forget the parametrisation of $V$.
First of all, by Corollary \ref{cor_dimensione}  we rewrite the results of  \cite{T} and \cite[Section 3]{T1}
as follows:

\begin{thm}\label{1ca}  Let $V \subseteq \bP^N$  be a  $k$-dimensional irreducible variety whose second fundamental form has dimension
$k-\ell-1$, with $\ell >0$. Then $V$ has tangent defect at least $\ell$ and it is contained in a scroll $S(\Sigma^h_t)$ in $\bP^t$ 
such  that
$T_{\bP^t_v}(S(\Sigma^h_t)) \subset \bP^{2k-h-\ell}$ with $0 \leq h \leq k-\ell$, where $v\in \Sigma^h_t$ is a general point, and
$\bP^t_v$ is the corresponding fibre of the scroll.
\end{thm}
\begin{thm}\label{2ca}
Let $V \subseteq \bP^N$  be a  $k$-dimensional irreducible variety.
  $V$ has tangent defect $o_1=\ell>0$ and the second fundamental form has dimension at least
$k-\ell$ if and only if the  Jacobian matrix of the second fundamental form of $V$ has rank $k-\ell$.
\end{thm}

We will prove Theorems \ref{mainthm} and \ref{thm terr gen}; Theorems \ref{1ca} and \ref{2ca} are just corollaries of them.

\begin{lem}\label{lemsbagliato}
Let $V \subseteq \bP^N$  be a  $k$-dimensional  variety and let $P\in V$. Then, the tangent cone to $\Tan^{t-1} (V)$ in $P$ is 
contained in   ${\tilde T}^{(t)}_P(V)$, and therefore 
$\tilde T_P(\Tan^{t-1} (V)) \subset  {\tilde T}^{(t)}_P(V)$.
\end{lem}

\begin{proof}
Let us take a frame on $V$ as above, i.e. such that $\{A_0;A_1,\dotsc,A_{k_t};\dotsc, A_N\}$ the first $k$-elements $A_1,\dotsc, A_k$ generate $T_P(V)$, and so on,
and therefore $A_1,\dotsc, A_{k_t}$ generate $T_P^{(t)}(V)$.

Let us take also a frame on $\Tan^{t-1} (V)$ centred at $P$,   $\{B_0;B_1,\dotsc,B_{\ell};\dotsc, B_N\}$, such that $B_0$ represents  $P\in\Tan^{t-1} (V)$ and
$B_1,\dotsc,B_{\ell}$ generate $T_P(\Tan^{t-1} (V))$.
By definition, we have
\begin{equation*}
B_0=C_0A_0+\sum_{i=1}^{k_{t-1}}C_i A_i;
\end{equation*}
taking the exterior derivative
\begin{equation*}
dB_0=d C_0 A_0+C_0 d A_0+ \sum_{i=1}^{k_{t-1}}(d C_i A_i+C_i d A_i),
\end{equation*}
from which we infer that the tangent cone to $\Tan^{t-1} (V)$ in $P$ is 
contained in   ${\tilde T}^{(t)}_P(V)$; since the tangent cone spans the tangent space, we conclude that 
$\tilde T_P(\Tan^{t-1} (V)) \subset  {\tilde T}^{(t)}_P(V)$. 
\end{proof}

\begin{thm} \label{mainthm} Let $V \subseteq \bP^N$  be a  $k$-dimensional irreducible variety whose $t$-th fundamental form has dimension $k-\ell-1$,
with $\ell >0$. Then:
\begin{enumerate}[1.]
  \item\label{primo} $V$ has $(t-1)$-osculating defect $o_{t-1}\geq \ell$.
  \item\label{secondo} $V$ is contained in a $d$-dimensional scroll $S(\Sigma^h_r)$, $(d\le h+r)$, in linear spaces of dimension $r$,
with $0 \leq h \leq k-\ell$ and $k-h\le r$.
  \item\label{terzo} Let $\ \bP^r\subset \ S(\Sigma^h_r)\ $ be a general $\ r\ $- dimensional space of the scroll;
then $\ \langle \cup_{A\in \bP^r}\tilde T_A( S(\Sigma^h_r))\rangle$ is contained in a linear space of dimension $d_t-h=d_{t-1}+ k-\ell-h
(\le\binom{k+t-1}{t-1}-1+k-\ell-h)$.
In particular, $r\le d\le d_{t-1}+ k-\ell-h$.
\end{enumerate}

\end{thm}
\begin{proof}

\ref{primo}:
 By hypothesis, Lemma \ref{lemsbagliato} and Corollary \ref{cor_dimensione ff} (and with the above notations)
\begin{multline*}
\dim \Tan^{t-1} (V)\leq \dim T_P( \Tan^{t-1} (V))\leq \\
\leq\dim(\tilde T^{(t)}_P(V))=d_{t-1}+\Delta_t+1\le \expdim\Tan^{t-1} (V)-\ell.
\end{multline*}

\ref{secondo}: Let---as above---$\gamma^t\colon V\dasharrow \mathbb G (N,d_t)$ the $t$-th Gauss map.
Let $h:=\dim \im (\gamma^t)$, so that $k-h$ is the dimension of the general fibre of $\gamma^t$.
Let $\Phi_{k-h}(\Pi):=(\gamma^t)^{-1}(\Pi)$ be a general fibre; this is just the set of points $Q\in V$ for which $\Pi={\tilde T}^{(t)}_Q(V)$.
$\Phi_{k-h}(\Pi)$ generates a linear space $\bP^r$, $k-h\le r \le d_t$.
Let us consider the scroll over  $\im (\gamma^t)=:\Sigma^h_r$ of these spaces,
$S(\Sigma^h_r)$.  Just by definition, $V\subset S(\Sigma^h_r)$.

Let $\check V^{(t)}\subset {\bP^N}^*$ be the $t$-th dual variety variety;
we have
\begin{equation*}
\dim(\check V^{(t)})=h+N-1-\dim(\tilde T^{(t)}_P(V)).
\end{equation*}

Moreover, by Lemma \ref{lemma}, \eqref{a}, $\tilde T^{(t-1)}_P(V))\subset C^{(1)}_{t,P}(V)$, so that (by Corollary \ref{cor_dimensione ff})
\begin{equation}\label{eq:equazione}
    d_{t-1}\leq N-1-d_{t,1}=d_t-h=d_{t-1}+\delta_t+1-h=d_{t-1}+k-\ell-h,
 \end{equation}
    and therefore $h\leq k-\ell$.

\ref{terzo}: We have, by Lemma \ref{lemma}, \eqref{b}, $Q\in C_{t,P}^{(t)}(V)$ if $Q\in \Phi_{k-h}(\Pi)$, $\Pi=\tilde T^{(t)}_P(V)$.
Since $C_{t,P}^{(t)}(V)$ is a linear space, we have that $\langle\Phi_{k-h}(\Pi)\rangle=\bP^r \subset C_{t,P}^{(t)}(V)$,
and therefore $S(\Sigma^h_r)\subset \Car^t_t(V)$. Finally, apply Corollary \ref{cor:lem}, to get that, if $R\in \bP^r$ is a general point,
$\tilde T _R S(\Sigma^h_r)\subset C^{(t-1)}_{t,P}(V)$ and moreover, since $\dim (\check V^{(t)})=N-1-d_t+h$, we have
that
\begin{equation*}
\dim C^{(t-1)}_{t,P}(V)\le d_t-h=d_{t-1}+ k-\ell-h\le \binom{k+t-1}{t-1}-1+k-\ell-h.
\end{equation*}

\end{proof}

Let us see some applications of this theorem.

\begin{ex}\label{ex:1}
Clearly, when $h=0$ we have that $V$ is contained in a $\bP^{d_{t}}$. For example, this is the only possibility when $k=1$ i.e the case of the curves; but in
this case we can say even more: we have $\ell=1$, $k-\ell=0=h$ and from \eqref{eq:equazione} we deduce that the curve is contained in a $\bP^{d_{t-1}}$.
So, if the theorem holds for $k=1$ and $t=2$, $V=\bP^1$ and for $k=1$ and $t=3$, $V$ is a plane curve, etc.
\end{ex}

\begin{ex}
More generally, if $\ell=k$, $h=0=k-\ell$, and also in this case, thanks to \eqref{eq:equazione},
we deduce that $V$ is contained in a $\bP^{d_{t-1}}$. In particular, if the theorem holds for $t=2$, we deduce $V=\bP^k$.
\end{ex}

\begin{ex}
Let us pass to the next case $\ell=k-1$; in this case $h=0,1$. If $h=0<1=k-\ell$, thanks to \eqref{eq:equazione}, we infer that $d_{t}=d_{t-1}+1$. Hence
$V\subset \bP^{d_{t-1}+1}$ by Example \ref{ex:1}. For $t=2$, we deduce that $V$ is a hypersurface in a $\bP^{k+1}$.

If  $h=1=k-\ell$, again by \eqref{eq:equazione}, we infer that $d_{t}=d_{t-1}+1$. Since, $k-1\le r\le d_t-1$ for $t=2$, we have that  $k-1\le r\le d\le k$, but
we cannot have $r=k$, since otherwise we would have that $V=\bP^r= S(\Sigma^h_r)$ and for it we would have $h=0$. Therefore, $r=k-1$, $\Phi_{k-h}(\Pi)=\bP^{k-1}$
 and $V$ is a developable $\bP^{k-1}$-bundle.
\end{ex}

Our result generalising Theorem \ref{2ca} is the following.
\begin{thm} \label{thm terr gen}
Let $V \subseteq \bP^N$  be a  $k$-dimensional irreducible variety.
  $V$ has $t$-th osculating defect $o_t=\ell>0$ and the
$(t+1)$-th fundamental form has dimension at least
$k-\ell$ if and only if the  Jacobian matrix of the $(t+1)$-th fundamental form of $V$ has rank $k-\ell$.
 \end{thm}
\begin{proof}
Let us fix as usual a Darboux frame for $V$ as in \eqref{eq:darbu}; if $P \in V$ is a general point, then, by Definition \ref{t^t def},
\begin{equation*}
\tilde T^t_P(V)=\langle (\frac{d^{|I|}A_0}{d v_1^{i_1}\dotsc d v_{k}^{i_{k}}})_{|I|\le t}\rangle,
\end{equation*}
with the convention that $d^0A_0=A_0$,
therefore, we can fix a  Darboux frame $\{B_0;B_1,\dotsc,B_d;B_{d+1},\dotsc,B_N\}$ ($d:=\dim \Tan^{(t)} (V) = \expdim \Tan^{(t)}(V)-\ell$) for $\Tan^{(t)}(V)$ centred at $Q \in \tilde T^t_P(V)$,
 where
$B_1,\dotsc,B_d$ span $T_Q(\Tan^{(t)}(V))$, and so
\begin{equation}\label{eq:b}
B_0=A_o\sum_{|I|=t}\lambda^{(I)}\frac{d^{|I|}A_0}{d v_1^{i_1}\dotsc d v_{k}^{i_{k}}}.
\end{equation}

Saying $\dim \Tan^{(t)} (V) = \expdim \Tan^{(t)}(V)-\ell$  means that there are $\ell$ linearly independent linear homogeneous relations between
the (first) partial derivatives of $B_0$ with respect to the $v_j$'s and the $\lambda^{(I)}$'s:
\begin{align*}
\sum_{j=1}^ka_{\alpha,j}\frac{\partial B_0}{\partial v_j}+\sum_{|I|= t}a_{\alpha,I}\frac{\partial B_0}{\partial \lambda^{(I)}}&=0 & \alpha=1,\dotsc,\ell;
\end{align*}
then, by \eqref{eq:b},
\begin{align*}
\sum_{j=1}^ka_{\alpha,j}(\sum_{|I|= t}\lambda^{(I)}\frac{d^{|I|+1}A_0}{d v_j d \mathbf v^I})&\equiv 0 & \alpha=1,\dotsc,\ell,  & \mod T^{(t)}_P(V)
\end{align*}
i.e. these relations are indeed relations between the partial derivatives up to order $t+1$ of $A_0$, and we can think of it as a system
of Laplace equations of order $t+1$:
\begin{align*}
\sum_{j=1}^ka_{\alpha,j}(\sum_{|I|= t}\lambda^{(I)}\mathbf x^{I+j}) &=0 & \alpha=1,\dotsc,\ell,
\end{align*}
and their associated polynomials are all reducible
\begin{align}\label{eq:apose}
(\sum_{j=1}^ka_{\alpha,j}v_j)(\sum_{|I|= t}\lambda^{(I)}\mathbf v_I)&=0 & \alpha=1,\dotsc,\ell,
\end{align}
with the same factor of degree $t$, $\sum_{|I|= t}\lambda^{(I)}\mathbf v_I$. Since these homogeneous polynomials are independent,
the  $\ell$ linear forms $\sum_{j=1}^ka_{\alpha,j}v_j$, $\alpha=1,\dotsc,\ell$, are independent. In particular, up to a change of coordinates,
it is not restrictive to suppose that these forms are $v_1,\dotsc,v_\ell$.

By Proposition \ref{prop_apolarity}, we have that the $(t+1)$-fundamental form is the apolar system associated to \eqref{eq:apose}; in particular, we have
that all the partial derivatives of the  $(t+1)$-fundamental form with respect to  $v_1,\dotsc,v_\ell$, are zero, from which we get that
the rank of the Jacobian is $k-\ell$.

Since all the above can be reverted, the vice versa easily follows.

\end{proof}

From this Theorem starts our next goal, i.e. to classify varieties with tangent or---more generally---higher osculating defect, which we will pursue in a subsequent
paper. This goal relies on the study
of linear systems of quadrics, or, more generally, of higher degree hypersurfaces, with Jacobian matrix of rank less than expected.




\begin{thebibliography}{99}

\bibitem{AG} M. A. Akivis and V. V. Goldberg, \emph{Differential Geometry of Varieties with Degenerate Gauss Maps},
CMS Books in Mathematics/Ouvrages de Math\'ematiques de la SMC, 18. Springer-Verlag, New York, 2004.






\bibitem{De1} L. Degoli, \emph{Sulle variet\`a $V_6$ i cui spazi tangenti ricoprono una variet\`a $W$ di dimensione inferiore all'ordinaria},
 Atti Sem. Mat. Fis. Univ. Modena, \textbf{16} (1967), 89--99.



%
%
%
%

%
%

\bibitem{G-H} P. Griffiths and J. Harris \emph{Algebraic geometry and local differential geometry},
Ann. Sci. \'Ecole Norm. Sup. (4), \textbf{12} (1979), no. 3, 355--452.

\bibitem{Harris}
J. Harris,
\newblock \emph{Algebraic geometry. A first course}, Number 133 in Graduate Texts in
  Mathematics. \newblock Springer-Verlag, New York, 1992.

\bibitem{H} R. Hartshorne, \newblock \emph{Algebraic Geometry}, \newblock Number~52 in Graduate Text in Mathematics. Springer-Verlag, New York, 1977.

\bibitem{IL} T. A. Ivey and J. M. Landsberg, \emph{Cartan for beginners: differential geometry via moving frames and exterior differential systems},
Graduate Studies in Mathematics, 61. American Mathematical Society, Providence, RI, 2003.

\bibitem{L} J. M. Landsberg,  \emph{On second fundamental forms of projective varieties}, Invent. Math. \textbf{117} (1994), no. 2, 303--315.

\bibitem{L1} J. M. Landsberg, \emph{Differential-geometric characterizations of complete intersections},
J. Differential Geom. \textbf{44} (1996), no. 1, 32--73.




\bibitem{LR} J. M. Landsberg and C. Robles, \emph{Lines and osculating lines of hypersurfaces}, J. Lond. Math. Soc. (2) \textbf{82} (2010), no. 3, 733--746.





\bibitem{M1} L. Muracchini, \emph{Le variet\`a $V_5$ i cui spazi tangenti ricoprono una variet\`a $W$ di dimensione inferiore alla ordinaria. II},
Rivista Mat. Univ. Parma, \textbf{3} (1952), 75--89.

\bibitem{M2} L. Muracchini, \emph{Le variet\`a $V_5$ i cui spazi tangenti ricoprono una variet\`a $W$ di dimensione inferiore alla ordinaria. I},
Rivista Mat. Univ. Parma, \textbf{2} (1951), 435--462.

\bibitem{M3} L. Muracchini, \emph{Le variet\`a $V_5$ i cui spazi tangenti ricoprono una variet\`a $W$ di dimensione inferiore alla ordinaria},
Boll. Un. Mat. Ital. (3), \textbf{6} (1951), 97--103.

\bibitem{pi} R. Piene,
\emph{A note on higher order dual varieties, with an application to scrolls}, Singularities, Part 2 (Arcata, Calif., 1981), 335--342,
Proc. Sympos. Pure Math., 40, Amer. Math. Soc., Providence, RI, 1983.


\bibitem{T} A. Terracini,
\emph{Sulle $V_k$ che rappresentano pi\`u di $\frac{1}{2} k (k - 1)$ equazioni di Laplace linearmente indipendenti},
Rend. Circ. Mat. Palermo, \textbf{33} (1912), 176--186.

\bibitem{T1} A. Terracini,  \emph{Alcune questioni sugli spazi tangenti ed osculatori ad una variet\`a. I}, Atti R. Accad. Sc. Torino,  \textbf{49} (1914),
214--247.

\bibitem{T2} A. Terracini,  \emph{Alcune questioni sugli spazi tangenti ed osculatori ad una variet\`a. II}, Atti R. Accad. Sc. Torino,  \textbf{51} (1915--6),
695--714.

\bibitem{T3} A. Terracini, \emph{Alcune questioni sugli spazi tangenti ed osculatori ad una variet\`a. III},
Atti R. Accad. Sc. Torino,  \textbf{55} (1919--20), 480--500.



\end{thebibliography}
\end{document}